\documentclass[letterpaper, 10 pt, journal, twoside]{ieeetran}

\IEEEoverridecommandlockouts                              

\usepackage[noadjust]{cite}
\usepackage{graphicx} 
\usepackage{epstopdf}
\usepackage{subcaption}
\usepackage{amsmath}
\usepackage{amssymb}
\usepackage{amsthm}
\usepackage{algorithm}
\usepackage{algpseudocode}
\usepackage{changepage}
\usepackage{lipsum}

\algrenewcommand\algorithmicrequire{\textbf{Input:}}
\algrenewcommand\algorithmicensure{\textbf{Output:}}

\newcounter{tempEquationCounter}
\newcounter{thisEquationNumber}

\newtheorem{lemma}{Lemma}
\newtheorem{theorem}{Theorem}

\newtheorem{proposition}{Proposition}

\newtheorem{corollary}{Corollary}

\usepackage{color}
\newcommand{\pedro}[1]{{\color{black}#1}}

\usepackage{float}
\newfloat{algorithm}{t}{lop}

\usepackage{color}
\definecolor{wheat}{rgb}{0.96,0.87,0.70}

\title{\pedro{Statistical Consistency of Set-Membership Estimator for Linear Systems}}
\author{Pedro Hespanhol and Anil Aswani%
\thanks{This material is based upon work partially supported by the National Science Foundation under Grant CMMI-1847666.}
\thanks{Pedro Hespanhol and Anil Aswani are with the Department of Industrial Engineering and Operations Research, University of California, Berkeley 94720
        {\tt\small \{pedrohespanhol,aaswani\}@berkeley.edu}}%
}

\begin{document}
\newcounter{notecounter}
\pagestyle{empty}
\maketitle
\thispagestyle{empty}
\pagestyle{empty}

\pedro{\begin{abstract}
Suppose we can choose from a set of linear autonomous systems with bounded process noise, the dynamics of each system are unknown, and we would like to design a stabilizing policy. The underlying question is how to estimate the dynamics of each system given that measurements of each system will be nonsequential. Though seemingly straightforward, existing proof techniques for proving statistical consistency of system identification procedures fail when measurements are nonsequential. Here, we prove that the set-membership estimator is statistically consistent even when measurements are nonsequential. We numerically illustrate its strong consistency. 
\end{abstract}}


\section{INTRODUCTION}

\label{sec1}


\IEEEPARstart{L}{earning-based} control has seen a resurgence in the past few years \cite{aswani2013_automatica,aswani2016b,ouyang2017learning,abbasi2018regret} because of recent advances in system identification using machine learning and artificial intelligence.  When the system has unknown dynamics, it becomes paramount to identify the underlying dynamics so that an appropriate controller can be computed in order to make the system stable \cite{ljung1987system}. System identification has become a central field of research lying in between control and statistics.

Here, we consider a fully observed switched autonomous linear system with bounded process noise. Each linear system is unknown to us, but we control switching between different linear dynamics. This departs from the existing literature on switched system identification, where the switching control is fixed and must also be estimated \cite{vidal2008recursive,garulli2012survey,lauer2018hybrid}. Here, a control decision needs to be chosen together with the system identification. The dynamics for a single system may have a mix of stable or unstable modes and repeated eigenvalues. Identification can be done via estimation of the transition matrices \cite{kumar1990convergence,soderstrom2012discrete}, and identification of transition matrices for \emph{stable} systems has been studied \cite{ljung1987system,soderstrom1989system,basu2015regularized,zorzi2017sparse}.
 
The identification problem in our setup is particularly challenging because the switching can cause stability/instability independent of the eigenvalues of each linear system \cite{bengea2005optimal}. The study of system identification for unstable systems is not as prolific as work on the stable case. Existing work for the unstable case of identification of a single linear system requires strong assumptions on repeated eigenvalues in order to prove asymptotic convergence \cite{lai1985asymptotic}, derive associated limiting distributions of the estimates of the model parameters \cite{buchmann2007asymptotic,buchmann2013unified}, and in order to generalize the result to other classes of transition matrices \cite{lai1983asymptotic,nielsen2005strong,nielsen2006order}.


{\color{black}\subsection{Linear System Identification}

Recent work \cite{faradonbeh2018finite,simchowitz2018learning,simchowitz2019learning,oymak2018non} has shown the difficulty of identification for unstable linear systems when state observations are restricted to a single trajectory: Ordinary least squares (OLS) is statistically inconsistent when the dynamics have repeated unstable dynamics \cite{nielsen2008singular,phillips2013inconsistent}, and this causes poor estimation when the dynamics have unstable modes with close eigenvalues. This can be partly overcome using instrumental variables, but this cannot handle systems matrices with eigenvalues both inside and outside the unit circle \cite{phillips2013inconsistent}.

The set-membership estimator \cite{bertsekas1971recursive,milanese1991optimal,aswani2017statistics} exploits boundedness of the noise vector. This estimator has been studied in \cite{deller1994unifying,lin1998consistently} which provided a bounding ellipsoidal algorithm to obtain consistent estimators. Our work is related to previous studies where such estimators are applied, as in fault detection tests \cite{blesa2012robust}, regularized regression \cite{beck2007regularization}, robust estimation \cite{garulli2000conditional,tjarnstrom2002mixed}, and kernel-based methods \cite{chen2014system}. The work in \cite{ozay2011sparsification} provides a greedy algorithm that uses a set-membership estimator to identify input-output models.

\subsection{Contributions}

Our main contribution is to prove (strong) statistical consistency of the set-membership estimator for switched linear autonomous systems, where measurements are not sequential and the system modes may be unstable. In past work, either the measurements were assumed to be sequential or statistical consistency was not proved. We use the idea behind Wald's Theorem \cite{wald1949note} to develop a novel consistency proof, in a way not done in other works \cite{milanese1991optimal,beck2007regularization}; however, Wald's Theorem itself does not apply to set-membership estimation, which imposes one constraint for each measurement, and only holds for estimators that minimize a lower semicontinuous loss.

To show a setting where the set-membership estimator is useful, we present a control policy that uses this estimator on a switched linear system. Our policy is a greedy bandit algorithm that uses the set-membership estimator to identify in finite-time the stable mode of the linear system. Our analysis is similar to recent work on greedy bandits \cite{faradonbeh2018finite,simchowitz2018learning}. The key difference in our setting is the state observations for each controller are not sequential, and so this means that OLS is not consistent for the matrix estimates in this setting.}

\subsection{Outline}

Sect. \ref{sec2} defines our notation, and Sect. \ref{Sec3} defines our problem setup. In Sect. \ref{sec4} we provide our proposed estimator and prove its statistical consistency. Next, in Sect. \ref{sec5} we numerically illustrate the strong consistency of the estimator.

\section{Notation}

\label{sec2}

We use $\|\cdot\|$ to denote the spectral norm of a matrix, which is the largest singular value of a matrix. We use the function $\rho(A)$ to denote the spectral radius of a matrix A. For a matrix $A$ we let $(A)_{ij}$ denote the $(ij)$-element of A. \pedro{For two sets $A$ and $B$, we denote their Minkowski sum by $A\oplus B$. Furthermore, the volume of set $A$ is $\mathrm{vol}(A)$. For a matrix $T$ and a set $A$, we define the set $TA := \{T\cdot a : a \in A\}$.} 

For matrix $A \in \mathbb{R}^{d \times d}$, let $v(A) \in \mathbb{R}^{d^2}$ be a vectorization that stacks elements of $A$ into a vector. For vector $u\in\mathbb{R}^{d^2}$, let $m(u) \in \mathbb{R}^{d\times d}$ be a matricization that folds elements of \pedro{$u$} into a matrix. We assume that $m\circ v(A) = A$ and $v \circ m(u) = u$.  Let \pedro{$\overline{\mathbb{R}}_{+} = \mathbb{R}_{+} \cup \{+\infty\}$} be the extended nonnegative real line. A function $f: \mathcal{D} \rightarrow \overline{\mathbb{R}}$ is lower semicontinuous (lsc) at $\overline{x}$ if and only if $\lim \inf_{x\rightarrow \overline{x}} f(x) \geq f(\overline{x})$. 

Next, we construct a compactification of $\mathbb{R}^n$ by defining \pedro{$\mathbf{A}^n = \mathbb{S}^{n-1}\times\overline{\mathbb{R}}_{+}$}, which directly compactifies \pedro{$\mathbb{S}^{n-1}\times{\mathbb{R}}_{+}$}. Note $\mathbf{A}^n$ can be shown to be equivalent to the \emph{cosmic closure} of $\mathbb{R}^n$, as defined in \cite{rockafellar2009variational}. To see why $\mathbf{A}^n$ is a compactification, note we can think of the $\mathbb{S}^{n-1} = \{v \in \mathbb{R}^n : \|v\|_2 = 1\}$ component as a direction of a vector and the \pedro{$\overline{\mathbb{R}}_+$} component as a length of the vector. Thus our idea is to formally use $\{\lambda v : (v,\lambda)\in\mathbf{A}^n\}$ as a compactification of $\mathbb{R}^n$.
\pedro{We define the expectation $\mathbb{E}[\cdot]$}. We use a.s. to denote ``almost surely'', and we use i.i.d. to denote ``independent and identically distributed''. 


\section{Problem Setup}

\label{Sec3}
Consider a fully observed switched linear system
\begin{equation} \label{LTI-SYS}
X_{t+1} = A_{\alpha_t} X_t + w_t
\end{equation}
where $X_t \in\mathbb{R}^d$ is the state, $w_t\in\mathbb{R}^d$ is the i.i.d. process noise, and $\alpha_t \in \{1,...,q\}$ is the control input that selects one of the (unknown to us) state dynamics \pedro{matrices} $A_1,...,A_q$. We assume $w_t$ lies in a (known to us) compact, convex set $\textbf{W} \subset \mathbb{R}^{d}$ that has a strict interior. Also, the $w_t$ has a (potentially unknown to us) p.d.f $w_t \sim f(w)$, where $\mathbb{E}[w_t] = 0$ and $f(w)>0$ for all $w \in \textbf{W}$; this assumption is mild for set-based estimation \cite{aswani2017statistics} and ensures the existence of a nonzero lower bound on the p.d.f.

Our goal is to estimate the matrices $A_1,...,A_p$, and we consider the situation where a subset of the matrices is unstable. In practical control applications, it is important to be able \pedro{to} precisely characterize the dynamics of each matrix so as to be able to design a stabilizing controller. Moreover, we wish to do the estimation without resetting the system (i.e., using a single state trajectory) and be able to do so given any arbitrary switching control input sequence $\{\alpha_t\}_{t \geq 0}$. 

Given an arbitrary (known to us) sequence of switching control inputs $\{\alpha_0,...,\alpha_{T-1}\}$ of length $T$, we collect the state measurements $\{x_{0},x_{1},...,x_{T}\}$. In order for the problem to be well-posed, we assume each linear system is selected at least $d$ times. Notationally, we organize measurements into groups where measurement pairs from the same linear system are grouped together: For each system $p$, we define the sequence of measurement pairs $\{(Y^{(p)}_i,X^{(p)}_i)\}_{i=1}^{n_p}$, where $n_p$ is the number of measurement pairs associated with system $p$. It is essential to note that for any $p$, a pair $(Y^{(p)}_i,X^{(p)}_i)$ is composed of successive observations of the system
\begin{equation}
Y^{(p)}_i = A_pX^{(p)}_i + w_i.
\end{equation}
Note $\{(Y^{(p)}_i,X^{(p)}_i)\}_{i=1}^{n_p}$ are generally not successive since $X^{(p)}_{i+1},Y^{(p)}_i$ are usually not the same because there may be an arbitrary number of switches between observations. Past consistency proofs for unstable systems (see for instance \cite{lai1983asymptotic}) require sequential measurements: These proofs separate the state dynamics into stable and unstable modes and then invert the unstable modes so that all the necessary quantities in the proof remain finite. When there is arbitrary switching, it is no longer possible to separate stable and unstable modes.

\section{Proposed Estimator and Consistency Proof}

\label{sec4}

Nonsequential observations makes system identification more challenging than estimation of autoregressive models. One naive approach is to use OLS for each group of data. This approach is inconsistent for general $A$ matrices \cite{nielsen2008singular,phillips2013inconsistent}, specifically $A$ with multiple geometric roots in the eigenvalue structure of the unstable matrix. These issues with OLS are numerically illustrated in Sect. \ref{sec5}.  Here, we provide an estimator that uses the boundedness of the disturbance vectors to overcome past issues. We prove consistency by adapting a celebrated argument by Wald \cite{wald1949note}, which is substantially different than typical analysis \cite{faradonbeh2018finite,simchowitz2018learning,simchowitz2019learning}.

\subsection{Set-Membership Estimator}

We focus our analysis on a single group $p$, and so we drop the superscript for ease of notation. Let $\{(Y_i,X_i)\}_{i=1}^{n}$ be our sequence of measurements
We let the associated true dynamics matrix $A_p$ be labeled as $A_0$. Hence it follows that
\begin{equation}
Y_i = A_{0}X_i + w_i, \ \text{for } i \in \{1,...,n\}.
\end{equation}
\pedro{Once again, we note the measurements pairs $(Y_i,X_i)$ and $(Y_{i+1},X_{i+1})$ are neither independent nor consecutive (in time) for any $i$, in general}. We propose to estimate $A_0$ by the minimizer to
\pedro{\begin{equation}
\label{eqn:est}
\begin{aligned}
\widehat{A} \in \arg\min_{A \in \mathbb{R}^{d\times d}}\ &\textstyle\frac{1}{n}\sum_{i=1}^{n}l(X_i,Y_i,A)\\
\mathrm{s.t.}\ & Y_{i} - AX_{i} \in \textbf{W} ,\ \text{for } i \in \{1,...,n\}\\
\end{aligned}
\end{equation}}
where $l(\cdot)$ is a loss function. For example, we may choose $l(X_i,Y_i,A) = \|Y_{i} - AX_{i}\|^2_2$. Observe that when $l(\cdot) \equiv 0$ this simply becomes a feasibility problem. We will first prove consistency of the feasibility version of this problem, which will imply consistency for well-behaved loss functions.

This is a \emph{set-membership estimator} and has been studied from the deterministic perspective \cite{bertsekas1971recursive,milanese1991optimal}. It uses the \textit{a priori} knowledge that process noise belongs to a compact convex set, in order to enforce constraints associated with each measurement pair. Here, we prove statistical consistency for this estimator when applied to this general setting of nonsequential and non-independent sequence of measurements pairs. In particular, by compactifying the domain of the optimization problem we are able to analyze the estimator by considering the statistics at only a finite number of points.

\pedro{\subsection{Local Identifiability of Problem Setup}}

\pedro{We begin by explicitly writing the feasibility version of the estimation but over a compactified domain:
\begin{equation}
\label{eqn:fe}
\begin{aligned}
\widehat{A} \in \arg\min_{A}\ &\textstyle\frac{1}{n} \sum_{i=1}^{n} \delta_{\textbf{W}}( Y_{i} - A X_{i})\\
\mathrm{s.t.}\ & A \in \{\lambda\cdot m(v): (v,\lambda)\in\mathbf{A}^{d^2}\}
\end{aligned}
\end{equation}}
where we define $\delta_{\textbf{W}}: \mathbb{R}^n \rightarrow \overline{\mathbb{R}}$ to be the indicator
\begin{equation}
\delta_{\textbf{W}}(u) = \begin{cases}
0, &\text{if } u \in \textbf{W}\\
+\infty , &\text{otherwise}
\end{cases}
\end{equation}
\pedro{As discussed in the next subsection, this compactification is required for the proof technique we use.} For notation, let $L(X_i,Y_i,A) = \delta_{\textbf{W}}(Y_{i} - AX_{i})$. We also need to specify arithmetic \cite{rockafellar2009variational} for points $(v,+\infty)\in\mathbf{A}^{d^2}$. For any $(\overline{X},\overline{Y})$ and $\overline{A}\in\{\lambda\cdot m(v) : (v,\lambda)\in\mathbb{S}^{d^2-1}\times\{+\infty\}\}$, define 
\begin{equation}
\label{eqn:lscdefl}
L(\overline{X},\overline{Y},\overline{A}) = \liminf_{(X,Y,A)\rightarrow(\overline{X},\overline{Y},\overline{A})} L(X,Y,A).
\end{equation}
Next, for each subset $S\subseteq\mathbf{A}^{d^2}$ we define
\begin{equation}
\begin{aligned}
h(X,Y,S) = \inf\ &L(X,Y,A)\\
\mathrm{s.t.}\ & A \in \{\lambda\cdot m(v): (v,\lambda)\in S\}
\end{aligned}
\end{equation}
We begin by characterizing the function $L(X,Y,A)$.

\begin{lemma}
Function $L(X,Y,A)$ is lower semicontinuous. 
\end{lemma}
\begin{proof}
Fix $(\overline{X},\overline{Y})$ and choose $\overline{A}\in\mathbb{R}^{d\times d}$. The function $Y-AX$ is continuous, and $\delta_\textbf{W}(u)$ is lower semicontinuous \cite{rockafellar2009variational}. Thus $L(\cdot)$ is lower semicontinuous at $(\overline{X},\overline{Y},\overline{A})$ since $L(X,Y,A) = \delta_\textbf{W}\circ(Y-AX)$. Next fix $(\overline{X},\overline{Y})$ and choose any $\overline{A}\in\{\lambda\cdot m(v) : (v,\lambda)\in\mathbb{S}^{d^2-1}\times\{+\infty\}\}$. Lower semicontinuity holds at this point by the definition (\ref{eqn:lscdefl}).
\end{proof}

Next define the extended real-valued function 
\begin{equation}
V(A) = \begin{cases} 0, &\text{if } A=A_0\\
+\infty, &\text{otherwise}\end{cases}
\end{equation}
and define $E(S) = \inf_{(v,\lambda)\in S}V(\lambda\cdot m(v))$. Proving statistical consistency requires verifying that some \emph{identifiability condition} holds \cite{bickel2015mathematical}, which means the underlying distributions are such that incorrect estimates are detected by measurements. If we define the mapping
\begin{equation}
\begin{aligned}
&\textstyle B_{n}(A) = \frac{1}{n} \sum_{i=1}^{n} L(X_i,Y_i,A)\\
&\textstyle H_{n}(S) = \frac{1}{n} \sum_{i=1}^{n} h(X_i,Y_i,S)
\end{aligned}
\end{equation}
\pedro{then we can prove a local identifiability condition holds.}
\pedro{\begin{proposition}
\label{ass:ic}
For any $A$ there is an open neighborhood $O(v,\lambda)\subset\mathbf{A}^{d^2}$, where $(v,\lambda)\in\mathbf{A}^{d^2}$ satisfies $A = \lambda\cdot m(v)$, such that $\lim_{n\rightarrow\infty}H_n(O(v,\lambda)) = E(O(v,\lambda)) = V(A) \text{ a.s}$.
\end{proposition}}

\begin{proof}
Let $(v_0,\lambda_0) \in \mathbf{A}^{d^2}$ be such that $\lambda_0\cdot m(v_0) = A_0$. Then $h(X_i,Y_i,O(v_0,\lambda_0)) \equiv 0$ for any open neighborhood $O(v_0,\lambda_0)$. This means that we immediately get that $\lim_{n\rightarrow\infty}H_n(O(v_0,\lambda_0)) = 0 = E(O(v_0,\lambda_0)) \text{ a.s.}$ 

Now consider any $A \neq A_0$, and let $t_i$ be the time of measurement $i$ for $i \geq 2$. Note $Y_i-AX_i = (A_0-A)X_{t_i} + w_{t_i}$, and $X_{t_i} = A_{\alpha_{{t_i}-1}}X_{{t_i}-1} + w_{{t_i}-1}$. Thus 
\begin{equation}
    Y_i - AX_i = (A_0-A)(A_{\alpha_{{t_i}-1}}X_{{t_i}-1}+w_{{t_i}-1}) + w_{t_i}.
\end{equation}
Let $\kappa = \min_{w\in\textbf{W}}f(w)$, and note $\kappa > 0$. The distribution of $Y_i-AX_i$ has support $\textbf{W} \oplus (A_0-A)\textbf{W} \oplus Z_i$ for $Z_i =  (A_0-A)A_{\alpha_{{t_i}-1}}X_{{t_i}-1}$. The key observation is that $\oplus Z_i$ translates the set $\textbf{W} \oplus (A_0-A)\textbf{W}$. \pedro{Let $\textbf{N}(A) = \{(u,v) \in \mathbb{R}^d\times\mathbb{R}^d : u + (A_0-A)v = 0\}$, and define $\textbf{V}(\textbf{S},A) = (\textbf{S} \oplus \textbf{N}(A)) \cap (\textbf{W}\times\textbf{W})$. Furthermore, define $\textbf{V}_i(A) = \{(u,v) \in \textbf{W}\times\textbf{W} : u + (A_0-A)v + Z_i\notin \textbf{W}\}$.
Thus we have 
\begin{multline}
\mathbb{P}[L(X_i,Y_i,A) = +\infty|T] = \\
\textstyle\int_{x}\mathbb{P}[L(X_i,Y_i,A) = +\infty|X_{t_i-1}=x,T]g(x)dx \geq \\
\textstyle\int_x[\int_{(u,v)\in \textbf{V}_i(A)} \kappa^2 dudv]g(x)dx \geq \\
\textstyle\int_x[\int_{(u,v)\in \textbf{V}(\textbf{J}(A),A)} \kappa^2 dudv]g(x)dx \geq \\
\textstyle\int_{(u,v)\in\textbf{V}(\textbf{J}(A),A)}\kappa^2dudv := c(A)
\end{multline}
for any event $T$ independent of $(w_{t-1},w_t)$, where $\textbf{J}(A) \in \arg\min_{\textbf{S}\subseteq\textbf{W}\times\textbf{W}} \{\mathrm{vol}(\textbf{V}(\textbf{S},A))\ |\ \mathrm{vol}(\begin{bmatrix} \mathbb{I} & (A_0-A)\end{bmatrix}\textbf{S}) = \mathrm{vol}(\textbf{W} \oplus (A_0-A)\textbf{W}) - \mathrm{vol}(\textbf{W})\}$, and $g(\cdot)$ is the p.d.f. of $X_{t_i-1}$ conditioned on $T$. Now define $B'_n(A) = \frac{1}{n/2-1}\sum_{k=2}^{n/2}L(X_{2k-1},Y_{2k-1},A)$ and note that by construction $(w_{t_{2k}-1},w_{t_{2k}})$ is independent of all $(X_{2k'-1},Y_{2k'-1})$ for $k' < k$. Thus for $n \geq 2$ we have
\begin{multline}
\mathbb{P}(B_n(A) = 0) \leq \mathbb{P}(B_n'(A) = 0) =\\
\mathbb{P}[L(X_{2\lfloor n/2\rfloor-1},Y_{2\lfloor n/2\rfloor-1},A) = 0 | B_{n-1}'(A) = 0]\times\\
\mathbb{P}(B_{n-1}'(A) = 0) \leq(1-c(A))\cdot\mathbb{P}(B_{n-1}'(A) = 0) \leq \\
\ldots \leq (1-c(A))^{\lfloor{n/2}\rfloor-1}.
\end{multline}
Noting $\mathrm{vol}(\textbf{W} \oplus (A_0-A)\textbf{W}) > \mathrm{vol}(\textbf{W})$, since $A\neq A_0$ and $\textbf{W}$ has a strict interior, then Fredholm's theorem for linear algebra implies $\mathrm{vol}(\textbf{V}(\textbf{J}(A),A)) > 0$. Hence $c(A) > 0$ and $\sum_{n=2}^\infty (1-c(A))^{\lfloor n/2\rfloor-1} < +\infty$. Thus the Borel-Cantelli lemma implies $B_n(A) = 0$ only finitely often. This proves $\lim_{n\rightarrow \infty}B_{n}(A) = V(A) = +\infty \text{ a.s.}$

Consider the same $A \neq A_0$, and define $(v,\lambda)$ so $\lambda\cdot m(v) = A$. Then for an open neighborhood $O(v,\lambda)$ we have $\textbf{Z}(v,\lambda) = \cap_{(u,\mu)\in O(v,\lambda)}\textbf{V}(\textbf{J}(\mu\cdot m(u)),\mu\cdot m(u))$ and
\begin{multline}
\label{eqn:pl2}
\mathbb{P}[h(X_i,Y_i,O(v,\lambda)) = +\infty|T] \geq \\
\textstyle\int_{(u,v)\in\textbf{Z}(v,\lambda)}\kappa^2dudv := d(O(v,\lambda)).
\end{multline}
By the Monotone Convergence Theorem, the open neighborhood $O(v,\lambda)$ can be chosen so $(v_0,\lambda_0)\notin O(v,\lambda)$ and so $d(O(v,\lambda)) > 0$. By a similar argument as before we have that $\mathbb{P}(H_n(O(v,\lambda)) = 0) \leq (1-d(A))^{\lfloor{n/2}\rfloor-1}$. So since $\sum_{n=2}^\infty (1-d(A))^{\lfloor{n/2}\rfloor -1} < +\infty$, the Borel-Cantelli lemma implies $H_n(O(v,\lambda)) = 0$ only finitely often. This proves $\lim_{n\rightarrow \infty}H_n(O(v,\lambda)) = E(O(v,\lambda)) = +\infty \text{ a.s.}$}
\end{proof}

\pedro{The above proposition establishes a local identifiability condition for our setup, namely a setting with nonsequential measurements and linear dynamics. The key intuition is that for any matrix $A$ the sample average $H_n(\cdot)$ converges to its ``expectation'' $E(\cdot)$ on some open neighborhood of $A$.}

\pedro{\subsection{Strong Statistical Consistency}}

\pedro{We are now in a position to prove our main theorem, which adapts the argument from the classical Wald Consistency Theorem \cite{wald1949note} and relies on the compactification of the feasible region. To understand the intuiton of why we compactify, recall that one definition of a compact set is a set where each of its open covers has a finite subcover. This is important for proving statistical consistency because, when parameters being estimated belong to a compact set, it allows us to perform an analysis only at a finite number of points in order to understand the global behavior. Compactification is important because it enables us to exploit this insight.}
\begin{theorem}
The feasibility estimator (\ref{eqn:fe}) is strongly consistent, meaning $\lim_{n\rightarrow \infty} \widehat{A} = A_{0} \text{ a.s.}$ or equivalently that $\mathbb{P}(\lim_{n\rightarrow \infty} \widehat{A} = A_{0}) = 1$.
\end{theorem}

\begin{proof}
Fix an open neighborhood $U$ around the matrix $A_0$. Because $A_0\in\mathbb{R}^{d\times d}$, the set $U$ can be represented as $U = \{\lambda\cdot m(v) : (v,\lambda)\in S\}$ for some $S \subset \mathbb{S}^{d^{2}-1}\times\mathbb{R}_+$. Recalling the definition of $V(\cdot)$, we know there exists $\epsilon > 0$ such that $V(A) \geq 3\epsilon + V(A_0)$ for $A\in \mathbf{C}(S)$, where
\begin{equation}
\mathbf{C}(S) = \{\lambda\cdot m(v) : (v,\lambda)\in\mathbf{A}^{d^2} \setminus S\}.
\end{equation}
For the next step, consider any fixed point $(v,\lambda)$ in $\mathbf{A}^{d^2} \setminus S$. Let $\{N_{k}(v,\lambda)\}_{k\geq 1}$ be a sequence of open balls that shrink to $(v,\lambda)$ as $k \rightarrow \infty$. Since $L(X,Y,A)$ is lower semicontinuous, it follows from the definition of $h(\cdot)$ that $\lim_{k \rightarrow \infty} h(X,Y,N_{k}(v,\lambda)) = L(X,Y,\lambda\cdot m(v))$. Since $A \in \mathbf{A}^{d^2} \setminus S$, the Monotone Convergence Theorem says there is an open neighborhood $N(v,\lambda)\subseteq O(v,\lambda)$ with
\begin{equation}
\label{eqn:uwp}
E(N(v,\lambda)) \geq V(A) - \epsilon \geq V(A_0) + 2\epsilon.
\end{equation}
\pedro{Now, since $\mathbf{A}^{d^2} \setminus S$ has been compactified then by one definition of a compact set there exists} a finite subcover $\mathcal{B}_1,\ldots,\mathcal{B}_z$ of neighborhoods $N(v,\lambda)$ centered around $(v_1,\lambda_1),\ldots,(v_z,\lambda_z)$. This means $\mathbf{A}^{d^2} \setminus S \subseteq \bigcup_{k=1}^z \mathcal{B}_k$, and 
\begin{equation}
\label{eqn:uwp2}
\textstyle\inf_{A \in \mathbf{C}(S)}B_{n}(A)  \geq \min_{k} \frac{1}{n} \sum_{i=1}^{n} h(X_i,Y_i,\mathcal{B}_k).
\end{equation}
\pedro{Using Proposition \ref{ass:ic} with (\ref{eqn:uwp}) and (\ref{eqn:uwp2}) implies 
\begin{equation}
\textstyle\lim_{n\rightarrow\infty}\inf_{A \in \mathbf{C}(S)} B_{n}(A) \geq V(A_0) + 2\epsilon \text{ a.s.}
\end{equation}}
By definition of (\ref{eqn:fe}) and $B_n(\cdot)$, $\widehat{A}$ minimizes $B_n(\cdot)$; hence, for almost all sample paths $\omega$ it follows that there exists $N$ such that for all $n > N$ we have
\begin{equation}
\textstyle B_{n}(\widehat{A}) \leq B_{n}(A_0) < V(A_0) + \epsilon < \inf_{A \in \mathbf{C}(S)} B_{n}(A).
\end{equation}
This implies that $\widehat{A}\in U$ for all $n > N$. We complete the proof by letting the neighborhood $U$ shrink to $\{A_{0}\}$.
\end{proof}
The above theorem proves consistency of the feasibility estimator (\ref{eqn:fe}). Consistency of the general estimator (\ref{eqn:est}) follows as a direct corollary for well-behaved loss functions.
\begin{corollary}
\label{cor:fo}
Suppose the loss function $l(X,Y,A)$ is continuous. Then the general estimator (\ref{eqn:est}) is strongly consistent, meaning $\lim_{n\rightarrow \infty} \widehat{A} = A_{0} \text{ a.s.}$.
\end{corollary}
\begin{proof}
Since $l(X,Y,A)$ is continuous, any $A$ feasible for (\ref{eqn:est}) is feasible for (\ref{eqn:fe}). Also, $A_0$ is feasible for (\ref{eqn:est}).
\end{proof}

\section{Numerical Experiments}

\label{sec5}

We demonstrate consistency of our estimator (\ref{eqn:est}) through two experiments. The first compares (\ref{eqn:est}) to OLS on identification for a dynamics matrix where OLS is inconsistent. The second uses (\ref{eqn:est}) to construct a switching control policy that identifies the stable mode of a switched linear system.

\subsection{Comparison to OLS}

Our first numerical experiment uses a single (i.e., no switching) state dynamics matrix that is given by
\begin{equation}
\label{eqn:a2}
A_{2} = \begin{bmatrix}
0\hphantom{.00} & 1.1\hphantom{0} & 0\hphantom{.00} & 0\hphantom{.00} \\ 1.1\hphantom{0} & 0\hphantom{.00} & 0\hphantom{.00} & 0\hphantom{.00} \\ 0\hphantom{.00} & 0\hphantom{.00} & 1.1\hphantom{0} & 0\hphantom{.00} \\ 0\hphantom{.00} & 0\hphantom{.00} & 0\hphantom{.00} & 1.1\hphantom{0}
\end{bmatrix} 
\end{equation}
This matrix is unstable since it has $\rho(A_2) = 1.1$. Moreover, the eigenvalue 1.1 has a geometric multiplicity of \pedro{three}. This means OLS is inconsistent when estimating $A_2$ from $X_t$ even in the absence of switching \cite{nielsen2008singular,phillips2013inconsistent}. In contrast, our estimator (\ref{eqn:est}) is consistent by Corollary \ref{cor:fo}. This is verified by Fig. \ref{fig:comp}, which shows results of a simulation with process noise that has uniform distribution with support $\textbf{W} = [-1,1]^4$. The estimation error of OLS remains nonzero, whereas the estimation error of (\ref{eqn:est}) using the loss function $l(X_i,Y_i,A) = \|Y_{i} - AX_{i}\|^2_2$ rapidly converges towards zero.

\begin{figure}
\centering
\includegraphics[scale=0.95]{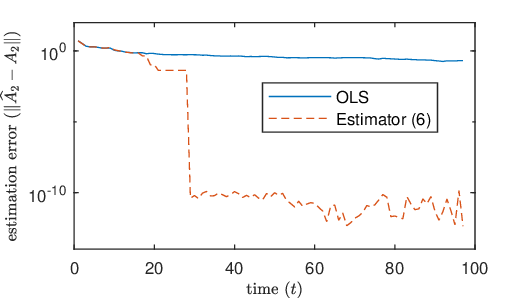}
\caption{Estimation Error From Trajectory by $A_2$ Without Switching}
\label{fig:comp}
\end{figure}

\subsection{Greedy Bandit Policy}

We next consider the setup in Sect. \ref{Sec3}, constrained so that there exists $s\in\{1,\ldots,q\}$ with $\rho(A_s) < 1$ and $\rho(A_p) >1$ for all $p\in\{1,\ldots,q\}\setminus\{s\}$. We specifically exclude the case $\rho(A_p) = 1$. Though (\ref{eqn:est}) is consistent when $\rho(A_p) = 1$, the policy we construct requires this assumption. We construct a policy that inputs the sequence $X_0,\ldots,X_t$ and $\alpha_0,\ldots,\alpha_{t-1}$ and chooses a control action $\alpha_t\in\{1,\ldots,q\}$ that identifies the stable mode while maintaining stability of the closed-loop system. This problem can be interpreted as a multi-armed bandit \cite{lai1985asymptotically,abbasi2011online,mintz2017non}, which involves a tradeoff between choices that: explore to learn more about the relevant distributions, and exploit by choosing the optimal (according to current estimates) actions. However, under specific assumptions a greedy algorithm can be (asymptotically) optimal \cite{gittins1979bandit,mersereau2009structured}.



Our procedure is Algorithm \ref{Alg}, and we use the loss function $l(X_i,Y_i,A) = \|Y_{i} - AX_{i}\|^2_2$ for (\ref{eqn:est}). We wish to identify the stable dynamics in finite time, because then the system can be brought to a stochastic equilibrium by selecting only the stable dynamics. The key idea is to use our estimator, which is consistent for all possible structures of $A$, once we group measurements as discussed in Sect. \ref{Sec3}. Note this algorithm greedily selects an arm with estimated spectral radius strictly smaller than 1. If at any given time $t$, no such arm exists, then we randomly select an arm and update the estimates. We can prove this algorithm maintains closed-loop stability:

\begin{figure}[t]
  \makebox[\linewidth]{%
  \begin{minipage}{\linewidth}
    \begin{algorithm}[H]
    	\caption{Greedy Bandit Algorithm}
    	\label{Alg}
    	\begin{algorithmic}[1]
    		\Require set $\{1,...,q\}$ of candidate systems. initial state $X_0$
            \For {systems $p \in \{1,...,q\}$:}
	    		\State	select system $p$
					\State obtain new measurement $X^{(p)}_{(1)}$
					\State set $n_p \leftarrow 1$
	    		\State	compute estimate $\widehat{A}_{p}$ using (\ref{eqn:est})
	    		\State compute estimate of spectral radius: $\hat{\rho}_{p} = \rho(\widehat{A}_{p})$
	    	\EndFor
	    	\For {each time instant $t > q$:}	
	    		\If {$\min_{p}\{\hat{\rho}_{p}\}\geq 1$} 
	    				\State randomly select a system $p$ 
							\State obtain new measurement $X^{(p)}_{(n_p+1)}$
							\State set $n_p \leftarrow n_p + 1$
	    			 	\State	compute estimate $\widehat{A}_{p}$ using (\ref{eqn:est})	\State compute estimate of spectral radius: $\hat{\rho}_{p} = \rho(\widehat{A}_{p})$
	    		\Else
	    			\State select any system $p$ such that $\hat{\rho}_{p} < 1$.
						\State obtain new measurement $X^{(p)}_{(n_p+1)}$ 
						\State set $n_p \leftarrow n_p + 1$
	    			\State compute estimate $\widehat{A}_{p}$ using (\ref{eqn:est})		  	\State compute estimate of spectral radius: $\hat{\rho}_{p} = \rho(\widehat{A}_{p})$
	    		\EndIf    		
    		\EndFor
    	\end{algorithmic}
    \end{algorithm}
\end{minipage}
}
\end{figure}

\begin{proposition}
Algorithm \ref{Alg} chooses the dynamics matrix $A_s$ infinitely many times and chooses the dynamics matrices $A_p$ for $p\in\{1,\ldots,q\}\setminus\{s\}$ only finitely many times.
\end{proposition}

\begin{proof}
We prove this by contradiction. Suppose there is a $p\in\{1,\ldots,q\}\setminus\{s\}$ such that the unstable dynamics $A_p$ is chosen infinitely many times. Since spectral radius is a continuous function \cite{kato2013perturbation}, combining the Continuous Mapping Theorem \cite{bickel2015mathematical,aswani2017statistics} with Corollary \ref{cor:fo} implies $\lim_{n\rightarrow\infty}\hat{\rho}_p = \rho_p > 1 \text{ a.s.}$; hence $\hat{\rho}_p < 1$ only finitely many times. Thus by construction of the algorithm, this means $A_p$ can be chosen by line 16 of the algorithm only finitely many times. So if $A_p$ is chosen infinitely often, this means it must be chosen by line 10 infinitely often. However, if this occurs then we must have that $\hat{\rho}_s > 1$ infinitely often.  However, again combining the Continuous Mapping Theorem with Corollary \ref{cor:fo} implies $\lim_{n\rightarrow\infty}\hat{\rho}_s = \rho_s < 1 \text{ a.s.}$ This is a contradiction.
\end{proof}

We conducted a numerical simulation to demonstrate the stabilizing behavior of our Algorithm \ref{Alg}. In the scenario we simulated, the process noise had a uniform distribution with support $\textbf{W} = [-1,1]^4$. In addition to $A_2$ as defined in (\ref{eqn:a2}), we used the state dynamics matrices
\begin{equation}
A_{1} = \begin{bmatrix}
0.76 & 0\hphantom{.00} & 1.6\hphantom{0} & 1.6\hphantom{0} \\ 0\hphantom{.00} & 0.78 & 0\hphantom{.00} & 1.6\hphantom{0} \\ 0\hphantom{.00} & 0\hphantom{.00} & 0.79 & 0\hphantom{.00} \\ 0\hphantom{.00} & 0\hphantom{.00} & 0\hphantom{.00} & 0.79
\end{bmatrix} 
\end{equation}
\begin{equation}
A_{3} = \begin{bmatrix}
0.91 & 0.7\hphantom{0} & 0\hphantom{.00} & 0\hphantom{.00} \\ 0.7\hphantom{0} & 0\hphantom{.00} & 0\hphantom{.00} & 0\hphantom{.00} \\ 0\hphantom{.00} & 0\hphantom{.00} & 0.28 & 0\hphantom{.00} \\ 0\hphantom{.00} & 0\hphantom{.00} & 0\hphantom{.00} & 1.05
\end{bmatrix} 
\end{equation}
\begin{equation}
A_{4} = \begin{bmatrix}
0\hphantom{.00} & 0\hphantom{.00} & 0.98 & 0\hphantom{.00} \\ 0\hphantom{.00} & 0\hphantom{.00} & 0\hphantom{.00} & 0.77 \\ 0.98 & 0\hphantom{.00} & 0.56 & 0\hphantom{.00} \\ 0\hphantom{.00} & 0.84 & 0\hphantom{.00} & 0.14
\end{bmatrix} 
\end{equation}
Note $\rho(\bar{A}_1) = 0.7900$, $\rho(\bar{A}_2) = 1.1000$, $\rho(\bar{A}_3) = 1.2899$, and $\rho(\bar{A}_4) = 1.2992$. This means $A_1$ is Schur stable while the other matrices $A_2,A_3,A_4$ are not Schur stable. However, $\|\bar{A}_1\| = 2.9136$, whereas $\|\bar{A}_2\| = 1.1000$, $\|\bar{A}_3\| = 1.2899$, and $\|\bar{A}_4\| = 1.2992$. This shows the importance of working with the spectral radius rather than using the spectral norm.

Numerical results of one simulation run are shown in Fig. \ref{fig:onesim}. Our other simulation runs had behavior that was qualitatively similar to the results we present here. At the beginning, the algorithm tries different arms. After a certain amount of tries of the different arms, the algorithm is able to identify which arm corresponds to the stabilizing mode. When the algorithm is trying different arms, the state grows at an exponential rate; however, once the stabilizing arm is found then the state fluctuates about the origin because of the process noise and the stabilizing action of that arm. 

\begin{figure}[t]
	\begin{center}
		\begin{subfigure}[t]{\linewidth}
			\includegraphics[scale=0.93]{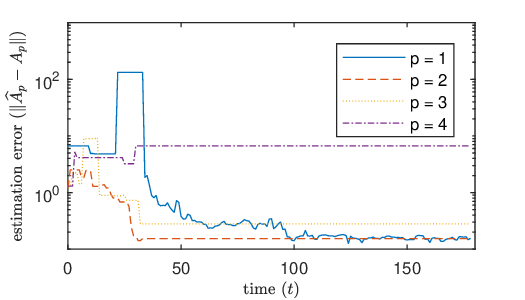}
			\caption{Estimation Error Using Our Estimator (\ref{eqn:est})}
		\end{subfigure}\\
		\begin{subfigure}[t]{\linewidth}
			\includegraphics[scale=0.93]{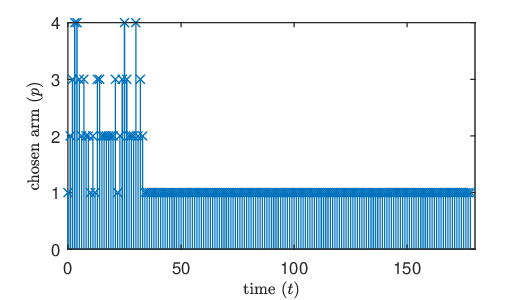}
			\caption{Arm $p$ Chosen by Algorithm}
		\end{subfigure}\\
		\begin{subfigure}[t]{\linewidth}
			\includegraphics[scale=0.93]{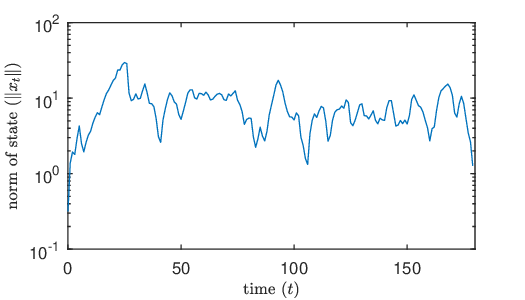}
			\caption{Norm of System State}
		\end{subfigure}
	\end{center}
	\caption{\label{fig:onesim} One Simulation Run of Algorithm \ref{Alg}}
\end{figure}

\section{Conclusion}

\label{sec6}


We proved statistical consistency of the set-membership estimator for identification of switched linear systems, and we demonstrated its consistency  through two numerical examples.

\bibliographystyle{IEEEtran}
\bibliography{IEEEabrv,secure}

\end{document}